\documentclass[graybox]{svmult}

\usepackage{type1cm}        
\usepackage{makeidx}         
\usepackage{graphicx}        
\usepackage{multicol}        
\usepackage[bottom]{footmisc}
\usepackage{verbatim}

\usepackage{newtxtext}       %
\usepackage[varvw]{newtxmath}       
\usepackage{overpic}
\usepackage{tikz}
\usepackage{amscd,latexsym,subfigure,verbatim,epsfig,amsmath,supertabular}
\usepackage[graph,frame,poly,arc]{xy}
\usepackage{calc}
\usepackage{graphicx}
\usepackage{mathtools}
\usepackage{enumitem}
\usepackage[T1]{fontenc}

\makeindex             

\newtheorem{defn0}{Definition}[subsection]

\theoremstyle{definition}
\newtheorem{example0}[definition]{Example}

\newcommand{\KK}{{\mathbb K}}
\newcommand{\PP}{{\mathbb P}}

 \DeclarePairedDelimiter{\seq}{\lbrace}{\rbrace}

\DeclarePairedDelimiter{\abs}{\lvert}{\rvert}

\DeclareMathOperator{\Betti}{Betti}

\newcommand{\Tor}{\mathop{\rm Tor}\nolimits}

\newcommand{\Ext}{\mathop{\rm Ext}\nolimits}

\begin{document}
\title*{Betti tables forcing failure of the Weak Lefschetz Property}

\author{Sean Grate, Hal Schenck}
\institute{Sean Grate \at Department of Mathematics, Auburn University, \email{sean.grate@auburn.edu} \and Hal Schenck \at Department of Mathematics, Auburn University, \email{hks0015@auburn.edu}}
\maketitle

\begin{abstract}
\noindent 
We study the Artinian reduction $A$ of a configuration of points $X \subset \PP^n $, and the relation of the geometry of $X$ to Lefschetz properties of $A$.
Migliore \cite{migliore_geometry_2008} initiated the study of this connection, with a particular focus on the Hilbert function of $A$, and further results appear in work of Migliore--Mir\'o-Roig--Nagel \cite{migliore_monomial_2011}.
Our specific focus is on Betti tables rather than Hilbert functions, and we prove that a certain type of Betti table forces the failure of the Weak Lefschetz Property (WLP). The corresponding Artinian algebras are typically not level, and the failure of WLP in these cases is not detected in terms of the Hilbert function.
\end{abstract}

\section{Introduction}
Let $ R = \KK[x_1,\hdots,x_n] $ and let $ A = R/I $ be a standard graded Artinian algebra, where $ I \subseteq R $ is a homogeneous ideal.
We say that $ A $ has the \emph{Weak Lefschetz Property} (\emph{WLP}) if multiplication by a general linear form \vskip -.1in
\[
A_i \stackrel{\cdot l}{\longrightarrow} A_{i+1} 
\]
has full rank for all $ i \geq 0 $, so is either injective or surjective.
If $ \cdot l^k \colon A_i \rightarrow A_{i+k} $ is full rank for all $ i,k \geq 0 $, we say $ A $ has the \emph{Strong Lefschetz Property} (\emph{SLP}).
As $ A $ is Artinian, $ A_i = 0 $ for $ i \gg 0$, so there are only finitely many maps to check.

The quintessential problem is to characterize those Artinian algebras which do (or do not) have WLP.
Many of the approaches to this problem in the literature tackle this by imposing constraints on the ideal $ I $ used in the construction of the Artinian algebra $ A $; see \cite{migliore_survey_2013} for a nice overview of the field and open problems.

In general, the Hilbert function of $ A $ does not determine if an Artinian graded algebra has WLP, but nevertheless the study of the relationship between Hilbert functions and Lefschetz properties has been fruitful. For example, in \cite{stanley_hilbert_1978}, Stanley constructs an Artinian-Gorenstein algebra with non-unimodal $ h $-vector $ (1,13,12,13,1) $; failure of unimodality forces a concomitant failure of WLP. Further results along these lines appear in \cite{boij_level_2007}, which proves existence of monomial Artinian level algebras of codimension $ r \geq 3 $ also having non-unimodal $ h $-vector. 

In this note, we examine a finer invariant than the Hilbert series: the graded Betti numbers of $ A $, and we show that certain Betti tables force failure of WLP in situations where the failure is not guaranteed from behavior of the Hilbert function. 

This project began with the goal of probing the Weak Lefschetz Property from a geometric perspective. Specifically, given a pointset $ X \subset \PP^n $, does the geometry of the points determine if the associated Artinian reduction has the WLP? Some other natural questions motivate the results of this paper:
\begin{itemize}
    \item Is there ever a set of points $Y \subset X$ (and perhaps conditions on $|Y|$ relative to $|X|$) that guarantee failure or success of WLP? \vskip .05in
    \item Given a configuration of points $X$, if WLP fails, can we predict the degree in which WLP fails in terms of the geometry of $X$? \vskip .05in
        \item If a set of points has multiple ``subconfigurations" of points--for example, two susbets of points lying on different hypersurfaces $ {\bf V}(f) $ and $ {\bf V}(g) $--then how does this structure influence WLP? What if there are conditions on $ f $ and $ g $?
\end{itemize}

\subsection{History and previous results} We enumerate some of the recent results which are closest in spirit to our work in this note; the Lefschetz literature is vast and so the list is not exhaustive.
\begin{itemize}
    \item Proposition 5.3 of \cite{migliore_geometry_2008} gives the existence of a reduced, level set of points where the WLP fails at semi-specified degrees via liason theory. \vskip .05in
\item Corollary 5.4 of \cite{boij_level_2007} shows the existence of pointsets in $ \PP^n $ with non-unimodal $ \Delta h $-vector, hence their Artinian reductions fail to have WLP. \vskip .05in
\item For points on a rational normal curve, Theorem 4.3 of \cite{altafi_hilbert_2022} shows that an Artinian Gorenstein algebra related to the points has the SLP. \vskip .05in
\item A similar result to that above is proved in \cite{altafi_hilbert_2022} for points on a conic in $ \PP^2 $.
\end{itemize}

Unexpected curves also seem to play an important role in the interplay between geometry and Lefschetz properties of pointsets.
Theorem 7.5 of \cite{cook_line_2018} shows that the set of points dual to a line arrangement in $ \PP^2 $ has an unexpected curve of a specified degree if and only if the the Artinian algebra formed from the line arrangement fails SLP in a specified range and degree. Expanding on the influence of unexpected curves on WLP, the papers \cite{bauer_quartic_2020}, \cite{di_marca_unexpected_2020}, and \cite{harbourne_unexpected_2021} give many more examples of unexpected curves, providing a richer collection of geometries to explore.

Geproci sets provide another approach to understanding how the geometry of a set of points influences WLP.
From a geproci set, a dual set of linear forms is defined and one can ask if the ideal generated by certain powers of the linear forms has WLP.
In the case that the geproci set is an $ (a,b) $ grid, \cite{chiantini_configurations_2022} shows that WLP does not hold for most pairs $ (a,b) $. Because of this connection between geproci sets and Lefschetz properties, there is much recent work on this topic; see for example \cite{chiantini_sets_2021}, \cite{chiantini_configurations_2022}, \cite{chiantini_classification_2023}.

\subsection{Results of this note} We approach Lefschetz questions from a geometric standpoint: our starting point is a configuration of reduced points $X \subset \PP^n$ (with $n \geq 3$), and we then study the Lefschetz properties of the Artinian reduction of $I_X$ by a generic hyperplane. Our aim is to investigate how the geometry of the points manifests itself in the associated Artinian reduction; in particular, if it is possible to detect whether or not the Artinian reduction has WLP.

The pointsets we study do not typically correspond to level algebras.
A number of papers such as \cite{stanley_hilbert_1978}, \cite{migliore_hilbert_2007}, \cite{migliore_strength_2008}, \cite{migliore_monomial_2011}, and \cite{altafi_hilbert_2022} have studied the influence of the Hilbert function on Lefschetz properties. A novel aspect of our approach is that the finer data available in the Betti table can distinguish WLP properties for pointsets having the same Hilbert function. In particular, from the Betti table of a pointset $ X \subset \PP^n $, we describe a simple indicator which is sufficient (but far from necessary) to show that the Artinian reduction of a pointset fails to have WLP.

\subsection{Notation}
Throughout this note, $ X $ will denote a finite set of distinct points in $ \PP^n $ (with $ n\geq 3$), $S = \KK[x_0,\hdots,x_n]$ will be the coordinate ring of $ \PP^n $, and $\KK$ a field. It is well known that the characteristic of $\KK$ has a major impact on WLP; however the results in this note are characteristic independent. 
Let $I_X \subseteq S$ denote the ideal of $ X $, and form the  Artinian reduction $ A_X $ by quotienting by a generic linear form $L$, which by definition misses all the points of $X$. Hence $L$ is a nonzero divisor on $S/I_X$. Let $J_X$ denote the Artinian reduction of $I_X$, so that 
\[
 A(X) = A_X = S/(I_X,L) \simeq R/J_X.
 \]
 We write $ \Betti(X) $ to denote the \emph{Betti table} (or \emph{Betti diagram}) of $ A_X $ (or equivalently, of $ J_X $).
The entries of $ \Betti(X)$ are given by the ranks and shifts of the summands appearing in the minimal free resolution of $ J_X $ (or $I_X$, since $L$ is a nonzero-divisor the Betti table is unchanged), so $\Betti(X)$ is a table whose entry in the $j^{th}$ row and $i^{th}$ column is 
\[
 \beta_{i,i+j}= \dim_{\KK}\Tor_i^R(R/J_X,\KK)_{i+j}.
\]

\section{All but one point on a hypersurface}
Given a pointset $X \subset \PP^n$ (with $n \geq 3$), we will show that if nearly all of the points lie on a unique hypersurface, then $A_X$ fails to have WLP. We illustrate this fact first with an example.
\begin{example0}
    Consider the points $ X \subset \PP^3 $ given as the columns of the matrix
    \begin{equation*}
        \begin{pmatrix}
            1 & 0 & 0 & 1 & 1 & 0 & 1 & 0 \\
            0 & 1 & 0 & 1 & 0 & 1 & 1 & 0 \\
            0 & 0 & 1 & 0 & 1 & 1 & 1 & 0 \\
            0 & 0 & 0 & 0 & 0 & 0 & 0 & 1
        \end{pmatrix},
    \end{equation*}
    and let $ X_f $ denote the seven points obtained from the first seven columns of the matrix above.
    The Hilbert functions of the Artinian reductions are $\{1,2,3,1\}$ for $X_f$ and $\{1,3,3,1\}$ for $X$. For the Betti table of $A_X$, we have
    \begin{equation*}
        \Betti(A_X) = 
        \begin{array}{l|cccc}
             & 0 & 1 & 2 & 3 \\
            \hline
            0 & 1 & . & . & . \\
            1 & . & 3 & 3 & 1 \\
            2 & . & 3 & 4 & 1 \\
            3 & . & . & 1 & 1
        \end{array}.
    \end{equation*}
    Let $ q $ be the point given by the last column of the matrix above. We constructed $ X $ so that all of the points in $ X_f $ lie on the plane $ x_3 = 0 $ and $ q $ lies off of it.
    In this way, we have that 
    \[
    (I_{X_f})_1 = (x_3) ,\mbox{ and }I_{\seq{q}} = (x_0,x_1,x_2).
    \]
    \vskip -.1in
    \noindent Thus
    \vskip -.2in
    \begin{align*}
       (I_X)_2 &= (I_{X_f})_1 \cap (I_{\seq{q}})_1 = (x_3) \cap (x_0,x_1,x_2) = (x_0 x_3, x_1 x_3, x_2 x_3).
    \end{align*}
    Now we consider the Artinian reduction $ A_X $ of $ I_X $ by a general linear form 
    \vskip -.1in
    \[L = \sum_{i=0}^3 a_i x_i.
    \]
    \vskip -.1in
   \noindent We then have 
    \[ (I_X,L)_2 = (x_0 x_3, x_1 x_3, x_2 x_3, x_3 L)= (x_0 x_3, x_1 x_3, x_2 x_3, x_3^2).
    \] 
    It is easy to see that multiplication by a generic linear form of $A$ will not have full rank, as $x_3$ is in the kernel of the multiplication map, hence $A_X$ does not have WLP. This is the motivating example for the following theorem.
\end{example0}

\begin{theorem}
\label{thm:all-but-one}
    Let $ X_f \subset \PP^n $ ($ n \geq 3 $) be a set of distinct points lying on a hypersurface $ {\bf V}(f) $ for some polynomial $ f \in R $ with $ \deg(f) = d $, such that $ (I_{X_f})_d = (f) $ and $ (I_{X_f})_{<d} = 0 $.
    Choose $ q \not\in {\bf V}(f) $ such that if $ X = X_f \cup \seq{q} $, then $ (I_X)_d = 0 $ and $ \dim_{\KK}((I_X)_{d+1}) = n $. Then the Artinian reduction $ A_X $ does not have WLP.
    In particular, if $ l $ is a general linear form, then $ A_d \xrightarrow{\cdot l} A_{d+1} $ does not have full rank.
\end{theorem}
\begin{proof}
    By changing coordinates, we may assume that $ q = [1:0:\hdots:0]$ and that $V(x_0) \cap X = \emptyset$. 
    We then have $ (I_X)_{\leq d} = (f) $ and $ I_{\seq{q}} = (x_1,\hdots,x_n) $.
    So then
    \begin{equation*}
        (I_X)_{d+1} = (f) \cap (x_1,\hdots,x_n) = (x_1 f,\hdots,x_n f).
    \end{equation*}
    It is clear that $x_i f \in (I_X)_{d+1}$ for $1 \leq i \leq n$, and we know the reverse containment $(x_1 f,\hdots,x_n f) \subseteq (I_X)_{d+1}$ since $\dim_{\KK}(I_X)_{d+1} = n$ by assumption.
    By our choice of coordinates above, the linear form $L = x_0$ is a non-zero divisor, and $A_X = R/J_X $ is the Artinian reduction of $ X $. Let $ l = \sum_{i=1}^n a_i x_i \in R $ be a general linear form; abusing notation we write $f$ for the reduction of $f$ modulo $x_0$. Then 
    \begin{equation*}
        l \cdot f = \left( \sum_{i=1}^n a_i x_i \right) \cdot f = \sum_{i=1}^n a_i x_i \cdot f =0.
    \end{equation*}
    Thus $ A_d \xrightarrow{\cdot l} A_{d+1} $ has nontrivial kernel, and is not injective. Another way to see this is to extend $\{f\}$ to a basis $\{f,g_1,g_2,\hdots,g_s\} $ of $ A_d $, and fix a basis $\{g_1',\hdots,g_t'\}$ of $A_{d+1}$.
    The map $ A_d \xrightarrow{\cdot l} A_{d+1} $ can then be represented by the matrix
    \begin{equation*}
      \bordermatrix{ & f & g_1 & g_2 & \hdots & g_s \cr
        g_1' & 0 & . & . & \hdots & . \cr
        g_2' & 0 & . & . & \hdots & . \cr
        g_3' & 0 & . & . & \hdots & . \cr
        \vdots & \vdots & \vdots & \vdots & \ddots & \vdots \cr
        g_{t'} & 0 & . & . & \hdots & .
      }
    \end{equation*}
    Thus $ A_{d} \xrightarrow{\cdot l} A_{d+1} $ is not injective since it has a zero column. The map cannot be surjective since $ \dim_{\KK}(A_d) = \binom{n-1+d}{n-1} < \binom{n+d}{n-1} - n = \dim_{\KK}(A_{d+1}) $. Note that this inquality is indeed strict since we assumed $n \geq 3$.
    Therefore $ A_d \xrightarrow{\cdot l} A_{d+1} $ does not have full rank, and $ A $ does not have WLP.
\end{proof}

The hypotheses of Theorem~\ref{thm:all-but-one} ensure that the points $ X_f $ lie on a \emph{unique} hypersurface and that $ q $ lies off of $ V(f) $. The number of points in $ \PP^n $ needed to ensure the hypotheses of Theorem~\ref{thm:all-but-one} is one less than the number of degree $d$ hypersufaces, so $ \abs{X_f} = \binom{n+d}{d} - 1 $ points suffice (if chosen generically) to fix $ {\bf V}(f) $, and 
\[
\dim_{\KK}((I_{X_f})_d) = 1, \mbox{ and }\dim_{\KK}((I_{X_f})_{<d}) = 0.
\]
Now suppose $ q \not\in V(f) $ and let $ \tilde{X} = X_f \cup \seq{q} $.
Then $ \dim_{\KK}((I_{\tilde{X}})_d) = 0 $ by the choice of $ q $.
We can then choose $ m \coloneqq \binom{n+d}{d+1} - n $ distinct generic points on $ {\bf V}(f) \setminus X_f $ to get $ \dim_{\KK}((I_{\tilde{X}})_{d+1}) = n $.
To see this, note that $ \dim_{\KK}(R_{d+1}) = \binom{n+d+1}{d+1} $ and solve the equation $ \dim_{\KK}(R_{d+1}) - (\abs{\tilde{X}} + m) = n $, i.e. the equation
\begin{align*}
    \binom{n+d+1}{d+1} - \left( \binom{n+d}{d} + m \right) = n,
\end{align*}
for $ m $ and use the identity $ \binom{n+d+1}{d+1} = \binom{n+d}{d} + \binom{n+d}{d+1} $.
Let $ \hat{X} $ be the set of $ m $ extra points chosen, and set $ X = \tilde{X} \cup \hat{X} $.
Then $ \dim_{\KK}(I_X)_{d+1} = n $, as desired.

\section{Maximal Koszul tails force failure of WLP}
The examples that led to Theorem~\ref{thm:all-but-one} exhibit very similar Betti tables, leading us to the following definition. \vskip .1in
\noindent {\bf Definition 1}
    We say a Betti table $ B $ has an \emph{$ (n,d) $-Koszul tail} if it has an upper-left principal block of the form
    \begin{equation*}
        \begin{array}{l|cccccccc}
             & 0 & 1 & 2 & 3 & \hdots & n-2 & n-1 & n \\
            \hline
            0 & 1 & . & . & . & \hdots & . & . & . \\
            1 & . & . & . & . & \hdots & . & . & . \\
            \vdots & \vdots & \vdots & \vdots & \vdots & \vdots & \vdots & \vdots & \vdots \\
            d-1 & . & . & . & . & \hdots & . & . & . \\
            d & . & n & \binom{n}{2} & \binom{n}{3} & \hdots & \binom{n}{n-2} & n & 1
        \end{array}.
    \end{equation*}
    If $ B $ has an $ (n,d) $-Koszul tail and is the Betti table for an Artinian ring $\KK[x_1,\ldots,x_n]/I$, then we say $ B $ has a \emph{maximal}  $(n,d) $-Koszul tail.

\begin{example0}\label{MaxKtail}
    Consider the pointset $ X_f \subset \PP^4 $ lying on a unique hypersurface $ {\bf V}(f) $ with $ \deg(f) = 3 $ (this will consist of $ 34 $ points plus $ 31 $ extra points).
    Take $ 5 $ points $ X_Q $ lying off of $ {\bf V}(f) $, but on a codimension $ 3 $ linear space.
    Let $ X \coloneqq X_f \cup X_Q $.
    Then the Betti table of $ A_X $ is
    \begin{equation*}
        \begin{array}{l|ccccc}
             & 0 & 1 & 2 & 3 & 4 \\
            \hline
            0 & 1 & . & . & . & . \\
            1 & . & . & . & . & . \\
            2 & . & . & . & . & . \\
            3 & . & 3 & 3 & 1 & . \\
            4 & . & 44 & 111 & 90 & 20 \\
            5 & . & . & . & . & 3
        \end{array}.
    \end{equation*}
    So the Betti table of $ A $ has a $ (3,3) $-Koszul tail, and in this case, $ A $ has WLP.
    However, if we form $ A^* $ from $ A $ by quotienting by yet another generic linear form, i.e. $ A^* = A/(L') $ with $ L' \in A_1 $, then
    \begin{equation*}
        \Betti(A^*) = 
        \begin{array}{l|cccc}
             & 0 & 1 & 2 & 3 \\
            \hline
            0 & 1 & . & . & . \\
            1 & . & . & . & . \\
            2 & . & . & . & . \\
            3 & . & 3 & 3 & 1 \\
            4 & . & 15 & 27 & 12
        \end{array}.
    \end{equation*}
    The Betti table of $ A^* $ has a maximal $ (3,3) $-Koszul tail, and $ A^* $ fails WLP from degree $ 3 $ to degree $ 4 $.
    This suggests that an Artinian ring whose Betti table has a maximal $ (n,d) $-Koszul tail fails WLP, and motivates our next theorem. The example above then follows from Corollary 2. 
\end{example0}

\begin{theorem}
\label{thm:maximal-Koszul-tail}
    An Artinian algebra $A=\KK[x_1,\ldots, x_n]/I$ whose Betti table has a maximal $ (n,d) $-Koszul tail does not have WLP; the map $A_d \stackrel{\cdot l}{\rightarrow} A_{d+1}$ is not full rank. 
\end{theorem}

\begin{proof}
    By reducing to the argument used in Theorem~\ref{thm:all-but-one}, it suffices to show that a maximal $(n,d)$-Koszul tail forces
    \[
    I_{d+1} = f \cdot (x_1,\ldots,x_n) \mbox{ for some } f \in \KK[x_1,\ldots,x_n]_d.
    \]
    Write the last differential in the resolution as $d_n$. For $A$ Artinian, the maximal Koszul tail hypothesis is equivalent to having the degree $n+d$ component of $d_n$--that is, the last map in the top strand--equal to 
    \[
    \psi = [x_1,\ldots,x_n]^T.
    \]
    To see this, note that the cokernel of $d_n^T$ is $\Ext^n(R/I,R)$, and the cokernel of $\psi^T$ is a direct summand of $\Ext^n(R/I,R)$--this is because the only part of $d_n^T$ whose target is $R(n+d)$ is exactly $\psi^T$, and we know the only associated prime of $\Ext^n(R/I,R)$ is the maximal ideal. But now we may dualize again, and since $R/I$ is Artinian, \[
    \Ext^n(\Ext^n(R/I,R),R) \simeq R/I.
    \]
    In particular, dualizing the resolution of $R/I$ gives a free resolution of $\Ext^n(R/I,R)$, and so the top nonzero row of the Betti table is  a resolution of the cokernel of $[x_1,\ldots,x_n]$, which is the Koszul complex on $[x_1,\ldots,x_n]$. Therefore in degree $d+1$ we have that $I_{d+1} = f \cdot (x_1,\ldots,x_n) $, and we conclude as in the proof of Theorem~\ref{thm:all-but-one}. 
\end{proof}
Note that rather than pointsets, we can take more general quotients. The key is to pass through a pointset having a maximal Koszul tail, which means having a regular sequence of maximal length. In particular, we get the following 
\begin{corollary}
If $T=\KK[x_0, \ldots, x_n]/I$ is Cohen-Macaulay of dimension $m$, and $T$ has a maximal $(n-m,d)$-Koszul tail, then the Artinian reduction of $T$ fails WLP.
\end{corollary}
As Example~\ref{MaxKtail} shows, we are not limited to the Cohen-Macaulay situation:
\begin{corollary}
If $A=\KK[x_1,\ldots, x_{m+n}]/I$ is Artinian with an $(n,d)$-Koszul tail, and there exists a sequence of linearly independent linear forms $I_L = (l_1,\ldots, l_m)$ such that $A/I_L$ has the same top row Betti table as $A$, then $A/I_L$ fails to have WLP. 
\end{corollary}
\begin{example0}
    Consider the points $ X \subset \PP^3 $ given as the columns of the matrix
    \begin{equation*}
        \begin{pmatrix}
            1 & 0 & 1 & 0 & 0 & 0 \\
            0 & 1 & 1 & 0 & 0 & 0 \\
            0 & 0 & 0 & 1 & 0 & 1 \\
            0 & 0 & 0 & 0 & 1 & 1
        \end{pmatrix}.
    \end{equation*}
    So $S = \KK[x_0,x_1,x_2,x_3]$ and $I_X = (x_0 x_2, x_1 x_2, x_0 x_3, x_1 x_3, x_0^2 x_1 - x_0 x_1^2, x_2^2 x_3 - x_2 x_3^2)$. The six points lie on the two disjoint lines ${\bf V}(x_0,x_1)$ and ${\bf V}(x_2,x_3)$. Letting $ L$  be a generic linear form, and $A = S/(I_X,L)  = R/J_X$, the $h$-vector of $A$ is $(1,3,2)$ so the one-dimensional socle in degree one does not block WLP.        
     \noindent The Betti table of $ A $ is
    \begin{equation*}
        \begin{array}{l|cccc}
             & 0 & 1 & 2 & 3 \\
            \hline
            0 & 1 & . & . & . \\
            1 & . & 4 & 4 & 1 \\
            2 & . & 2 & 4 & 2
        \end{array}.
    \end{equation*}
    A computation with {\tt Macaulay2} verifies that $ A $ has WLP. This example highlights that even if there is a single minimal-degree element in the socle of $ A $, this does not imply that $ A $ fails WLP. It is the {\em maximal} $(n,d)$-Koszul tail that is the key: the map $A_d \stackrel{\cdot l}{\rightarrow} A_{d+1}$ cannot be surjective when $n\ge 3$. \end{example0}

\noindent A non-maximal $ (n,d) $-Koszul tail does not determine WLP.
In \cite{abdallah_free_2022}, the authors consider the Betti table
\begin{equation*}
    \begin{array}{l|ccccc}
     & 0 & 1 & 2 & 3 & 4 \\
     \hline
     0 & 1 & . & . & . & . \\
     1 & . & 2 & 1 & . & . \\
     2 & . & 5 & 9 & 4 & . \\
     3 & . & 4 & 9 & 5 & . \\
     4 & . & . & 2 & 1 & . \\
     5 & . & . & . & . & 1 \\
    \end{array}.
\end{equation*}
The following two ideals both yield the Betti table above, but the ideal 
\begin{align*}
    &(x_4^2, x_3 x_4, x_3^3, x_2 x_3^2 - x_2^2 x_4, x_1 x_3^2 - x_1 x_2 x_4 + x_2^2 x_4, x_2^2 x_3, x_2^3, \\
    &\hspace*{4cm} x_1^3 x_4 - x_1^2 x_2 x_4 + x_1 x_2^2 x_4, x_1^3 x_3, x_1^3 x_2 - x_1^2 x_2^2, x_1^4)
\end{align*}
has WLP, while the ideal 
\begin{equation*}
    (x_1 x_4, x_1^2, x_3 x_4^2, x_2 x_4^2, x_2^2 x_4, x_1 x_3^2, x_1 x_2^2 - x_3^2 x_4, x_3^4, x_2 x_3^3 - x_4^4, x_2^2 x_3^2, x_2^4)    
\end{equation*}
fails to have WLP.
\begin{remark}
All our computations assume we are starting in $\PP^n$ with $n \ge 3$. Artinian quotients of the form $\KK[x,y]/I$ always have WLP, so our technique does not apply to $(2,d)$-Koszul tails. 
\end{remark}

\section*{Acknowledgements}
We thank Juan Migliore and Uwe Nagel for useful feedback during the meeting ``Workshop on Lefschetz Properties in Algebra, Geometry, Topology and Combinatorics'', held at the Fields Institute in May 2023, and we thank the organizers of that workshop, as well as the organizers of the workshop ``The Strong and Weak Lefschetz Properties'', held in Cortona in September 2022, and an anonymous referee for helpful comments. Computations were preformed using {\tt Macaulay2} \cite{M2}. Schenck was supported by NSF 2006410 and the Rosemary Kopel Brown professorship.

\bibliographystyle{amsplain}
\bibliography{refs}

\end{document}